\documentclass[12pt,a4paper]{article}
\usepackage{amsfonts}
\usepackage{amsmath}
\usepackage{amsbsy}
\usepackage{amsxtra}
\usepackage{latexsym}
\usepackage{amssymb}
\usepackage{url}
\usepackage{pstricks,pst-node}
\usepackage{enumerate}
\makeatletter
\g@addto@macro\thesection.
\makeatother

\newtheorem{theorem}{Theorem}
\newtheorem{lemma}{Lemma}

\DeclareMathOperator{\tr}{tr}
\DeclareMathOperator{\diag}{diag}

\newcommand{\lng}{\langle}
\newcommand{\rng}{\rangle}
\newcommand{\lf}{\left}
\newcommand{\rg}{\right}

\newcommand{\R}{\mathbb R}

\newcommand{\an}{\mathbb A^n}

\newcommand{\snp}{\mathbb S^n_+}

\newcommand{\nn}{\mathbb R^{n\times n}}

\newcommand{\tp}{^\top}

\newenvironment{proof}{{\noindent\bf Proof.}}{\hfill$\Box$\\}

\begin{document}

\title{The dual Z-property for the Lorentz cone}
\author{S. Z. N\'emeth\\School of Mathematics, The University of Birmingham\\The Watson Building, Edgbaston\\Birmingham B15 2TT, United
Kingdom\\email: s.nemeth@bham.ac.uk}
\date{}
\maketitle

\begin{abstract}
	The Z-property of a linear map with respect to a cone is an extension of the notion of Z-matrices. In a recent paper of
	Orlitzky (see Corollary 6.2 in M. Orlitzky. Positive and {$\mathbf{Z}$}-operators on closed convex cones,
	\textit{Electron. J Linear Algebra}, 444--458, 2018) the characterisation of cone-complementarity is given in terms of the dual of the 
	cone of linear maps satisfying the Z-property. Therefore, it is meaningful to consider the
	problem of finding the dual cone of the cone of linear maps which 
	have the Z-property with respect to a cone. This short note will solve this problem in the particular case when the Z-property is
	considered with respect to the Lorentz cone.
\end{abstract}

\section{Introduction} 

In the followings we will identify a linear map $A:\,\mathbb R^n\to\mathbb R^n$ with its matrix in $\nn$ with respect
to the canonical basis. In $\R^N$ consider the standard inner product \[\lng
x,y\rng=x\tp y=x_1y_1+x_2y_2+\dots+x_Ny_N,\] for any two column vectors $x=(x_1,\dots,x_N)\tp$, $y=(y_1,\dots,y_N)\tp$. In the case of 
$N=n\times n$ this can be expressed by the standard trace formula $\lng
A,B\rng=\tr{A\tp B}$, for any $A,B\in\R^{n\times n}$. All following duals will be considered with respect to the
standard inner product. 

Recall that a closed set $K\subseteq\R^N$ is called a closed convex cone if and only $\lambda u\in K$ and $u+v\in K$ for 
any $u,v\in K$ and any $\lambda>0$. We remind that the dual of $K$ is the closed convex cone $K^*$ defined by 
\[K^*=\{z\in\R^N:\,\lng z,v\rng\ge 0,\textrm{ }\forall v\in K\}.\]

Denote $C(K)=\{(x,y)\in\R^n\times\R^n:\,x\in K,\textrm{ }y\in K^*,\textrm{ and }\lng x,y\rng=0\}$.
One says that a linear map $A:\,\mathbb R^n\to\mathbb R^n$ has the Z-property with respect to $K$ (or shortly $K$-Z-property) if 
$(x,y)\in C(K)$ implies $\lng Ax,y\rng\le 0$. This notion has been introduced by Gowda and Tao in \cite{GowdaTao2009}. It is easy to check that a
matrix $A$ has the Z-property with respect to the nonnegative orthant 
$\R^n_+$ if an only if it is a Z-matrix, that is, a matrix whose off-diagonal elements are nonpositive
(\cite{FiedlerPtak1962},\cite[Definition 2.5.1]{HornJohnson1991}). 

Z-matrices are important in differential equations, dynamical systems, 
optimization, economics etc. \cite{BermanPlemons1994}. The importance of Z-matrices in linear complementarity is clear from the equivalences 
(1)-(6) in the Introduction of \cite{GowdaTao2009} which come from Chapter 6 of \cite{BermanPlemons1994}. 

The following papers justify the importance of the Z-property in finite
dimensions: \cite{GowdaTao2009,Gowda2012,GowdaTaoRavindran2012,TaoGowda2013,KongTaoLuoXiu2015,Orlitzky2017,Orlitzky2018}.
Theorems 6, 7, 9, 10 and 13 in \cite{GowdaTao2009} show the crucial role of Z-property in linear cone-complementarity, extending the 
fundamental role of Z-matrices in linear complementarity. 

We also mention that important results have been obtained for the Z-property in infinite dimensions too \cite{FanTaoRavindran2017}. 

Denote by $Z(K)$ the cone of linear mappings with $K$-Z-property. 
The dual of $Z(K)$ appears first time in the paper \cite{Orlitzky2018} of Orlitzky.
From Corollary 6.2 of this paper it follows that $(x,y)\in C(K)$ if and only if
$-yx\tp\in Z(K)^*$, which can also be easily checked by using the properties of the trace function for matrices.  
Although simple, this equivalence makes the study of $Z(K)^*$ an interesting topic in cone-complementarity. In the case of the
nonnegative orthant $\R^n_+$ it is a straightforward exercise to check that the dual of $Z(\mathbb R^n_+)$ is
formed of nonpositive matrices with zero diagonal. In general finding $Z(K)^*$ is a highly nontrivial
task. This short note will determine $Z(L)^*$ for the Lorentz cone L.

\section{The main result}
Let $J=\diag(1,-1,-1,\dots,-1)\in\nn$, $L\subset\R^n$ be the Lorentz cone defined by
\[L=\{(x_1,x_2,\dots,x_n)\tp\in\R^n:\,x_1\ge\sqrt{x_2^2+\dots+x_n^2}\}\]
and $Z(L)$ the cone of linear maps $A:\,\mathbb R^n\to\mathbb R^n$ with $L$-Z-property (that is, Z-property with respect to $L$). 

Denote by $\mathbb S^n$, $\mathbb S^n_+$ and $\mathbb A^n$ the linear subspace of symmetric 
matrices of $\nn$, the cone of $n\times n$ symmetric positive semidefinite matrices and the linear subspace of skew-symmetric matrices of 
$\nn$, respectively.

Then, we have the following lemma:
\begin{lemma}\label{lz}
	\begin{equation*}
		Z(L)=\{\gamma I-J(P+Q):\,\gamma\in\R,\textrm{ }P\in\snp\textrm{ and } Q\in\an\}.
	\end{equation*}
\end{lemma}
\begin{proof}
	According to Gowda and Tao \cite[Example 4]{GowdaTao2009} we have 
	\begin{equation}\label{egt}
		Z(L)=\{A\in\nn:\,\exists\gamma\in\R\textrm{ such that }\gamma 
		J-(JA+A\tp J)\in\snp\}.
	\end{equation}
	Denote 
	\begin{equation*}
		\Gamma=\{\gamma I-J(P+Q):\,\gamma\in\R,\textrm{ }P\in\snp\textrm{ and } Q\in\an\}.
	\end{equation*}
	We need to show that $Z(L)=\Gamma$. 
	
	Let $A=\gamma I-J(P+Q)\in\Gamma$, where $\gamma\in\R$, $P\in\snp$ and
	$Q\in\an$. Then, $J(P+Q)=\gamma I-A$. Multiplying this equation by $J$ from the left,
	and using $J^2=I$, we obtain $P+Q=\gamma J-JA$. Hence, $Q=\gamma J-JA-P$ and thereforee
	$0=Q+Q\tp=2\gamma J-(JA+A\tp J)-P$. Thus, $2\gamma J-(JA+A\tp J)\in\snp$, which
	according to \eqref{egt} implies that $A\in Z(L)$. This yields $\Gamma\subseteq Z(L)$. 

	Conversely, let that $A\in Z(L)$. Then, according to \eqref{egt}, we get $2\gamma J-(JA+A\tp J)=2P$, for some
	$\gamma\in\R$ and some positive semidefinite matrix $P\in\nn$. Hence, $JA+(JA)\tp=2(\gamma
	J-P)$, which implies $JA=\gamma J-P-Q$, for some skew-symmetric matrix $Q\in\nn$.
	Multiplying the last equality by $J$ from the left and using $J^2=I$, we obtain that $A=\gamma I-J(P+Q)$, which
	implies $A\in\Gamma$. This yields $Z(L)\subseteq\Gamma$.

	In conclusion, $Z(L)=\Gamma$.
\end{proof}

In the next theorem we find the dual of $Z(L)$.
\begin{theorem}
	\[Z(L)^*=\{B\in\nn:\,\tr{B}=0\textrm{ and }JB\textrm{ is negative semidefinite}\}.\] 
\end{theorem}
\begin{proof} 
	By using Lemma \ref{lz} we have
	\begin{eqnarray*}
		Z(L)^*=\{B\in\nn:\tr(B\tp A)=\lng B,A\rng\ge 0,\textrm{ }\forall A\in Z(L)\}\\
		=\{B\in\nn:\,\tr\lf(B\tp[\gamma I-J(P+Q)]\rg)\ge0,\textrm{ }\forall\gamma\in\R,
		\textrm{ }\forall P\in\snp\textrm{ and }\forall Q\in\an\}\\
		=\{B\in\nn:\,\gamma\tr(B)-\tr\lf(B\tp J(P+Q)\rg)\ge0,\textrm{ }
		\forall\gamma\in\R,\textrm{ }\forall P\in\snp\textrm{ and }\forall Q\in\an\}\\
		=\{B\in\nn:\,\tr(B)=0,\textrm{ }\tr\lf(B\tp J(P+Q)\rg)\le0,
		\textrm{ }\forall P\in\snp\textrm{ and }\forall Q\in\an\}\\
		=\{B\in\nn:\,\tr(B)=0,\textrm{ }\tr\lf((P+Q)B\tp J\rg)\le0,
		\textrm{ }\forall P\in\snp\textrm{ and }\forall Q\in\an\}\\
		=\{B\in\nn:\,\tr(B)=0,\textrm{ }\lng (P+Q)\tp,B\tp J\rng\le0,
		\textrm{ }\forall P\in\snp\textrm{ and }\forall Q\in\an\}\\
		=\{B\in\nn:\,\tr(B)=0,\textrm{ }\lng P+Q,B\tp J\rng\le0,
		\textrm{ }\forall P\in\snp\textrm{ and }\forall Q\in\an\}.
	\end{eqnarray*}

	The last equality in
	the formula of $Z(L)^*$ above shows that $B\in Z(L)^*$ if and only if $\tr(B)=0$ and $-B\tp J\in (\mathbb S^n_++\mathbb A^n)^*$, where the dual is 
	considered as the dual of a cone in $\nn$. 

	In case of $\mathbb S^n_+$ we should make a distinction between the dual of this 
	cone in $\nn$ and the dual in $\mathbb S^n$. In case of the former we will use the notation $(\mathbb S^n_+)^*$ and in the case
	of the latter $(\mathbb S^n_+)_*$. 

	In general for a closed convex cone $K\subseteq\R^N$ (in our case $N=n\times n$) we denote by $K^*$ 
	the dual of $K$ in $\R^N$, by $K_*$ the dual of $K$ in the linear subspace $K-K$ spanned by $K$, and by $\perp$ the orthogonal 
	complement operation for subspaces. 

	It is an easy exercise to check that the dual of a subspace (as a cone) is its orthogonal complement
	and $K^*=K_*+(K-K)^\perp$. It is well known that $\mathbb S^n_+$ is self-dual in $\mathbb S^n$, that is $(\mathbb S^n_+)_*=\mathbb S^n_+$
	and that $(\mathbb S^n)^\perp=\mathbb A^n$. Hence, the membership $-B\tp J\in (\mathbb S^n_++\mathbb
	A)^*$, which together with $\tr(B)=0$ characterizes $B\in Z(L)^*$, is 
	equivalent to 
	\begin{eqnarray*}
		-B\tp J\in(\mathbb S^n_++\mathbb A^n)^*=(\mathbb S^n_+)^*\cap (\mathbb A^n)^*=(\mathbb S^n_+)^*\cap (\mathbb A^n)^\perp
		=\lf[(\mathbb S^n_+)_*+(\mathbb S^n_+
		-\mathbb S^n_+)^\perp\rg]\cap (\mathbb A^n)^\perp\\
		=\lf[\mathbb S^n_++(\mathbb S^n_+-\mathbb S^n_+)^\perp\rg]\cap (\mathbb A^n)^\perp
		=\lf[\mathbb S^n_++(\mathbb S^n)^\perp\rg]\cap (\mathbb A^n)^\perp
		=\lf[\mathbb S^n_++\mathbb A^n\rg]\cap\mathbb S^n=\mathbb S^n_+,
	\end{eqnarray*}
	or to $-JB\in\mathbb S_n^+$, as $-B\tp J=(-B\tp J)\tp=-JB$ and it is an easy exercise to check that for any two closed convex cones 
	$K_1$, $K_2$ we have $(K_1+K_2)^*=K_1^*\cap K_2^*$. In conclusion, 
	\[Z(L)^*=\{B\in\nn:\,\tr{B}=0\textrm{ and }JB\textrm{ is negative semidefinite}\}.\] 
\end{proof}

\bibliographystyle{habbrv}
\bibliography{zstar}

\end{document}